\documentclass[12pt,a4paper,reqno]{amsart}
\usepackage[T1]{fontenc}
\usepackage{microtype, a4wide, textcomp,fbb}
\usepackage{
    amsmath,  amssymb,  amsthm,   amscd,
    gensymb,  graphicx, etoolbox, 
    booktabs, stackrel, mathtools    
}
\usepackage[usenames,dvipsnames]{xcolor}
\definecolor{darkblue}{rgb}{0,0,0.7}
\definecolor{darkred}{rgb}{0.7,0,0}
\usepackage[colorlinks=true, linkcolor=darkred, citecolor=darkblue, urlcolor=blue, pagebackref=true, breaklinks=true]{hyperref}
\usepackage[capitalise]{cleveref}
\usepackage{placeins}
\usepackage{relsize}
\usepackage{todonotes}
\usepackage{soul}
\usepackage{enumitem}
\newtheorem{theorem}{Theorem}[section]

\newtheorem{corollary}[theorem] {Corollary}
\newtheorem{lemma}[theorem]{Lemma}

\newtheorem{proposition}[theorem]{Proposition}
\theoremstyle{definition}

\newtheorem{example}[theorem]{Example}

\newtheorem*{problem*}{Problem}
\newtheorem{remark}[theorem]{Remark}

\usepackage{tikz}
\usetikzlibrary{arrows}
\usetikzlibrary{shapes}
\usepackage{float}
\usepackage{tabularx}
\usepackage{kantlipsum}
\usepackage{array}
\usepackage[margin=2.6cm]{geometry}
\tikzset{edgee/.style = {->,> = latex'}}
\newcolumntype{P}[1]{>{\centering\arraybackslash}p{#1}}
\newcolumntype{M}[1]{>{\centering\arraybackslash}m{#1}}

\newcommand{\G}{\mathcal{G}}
\newcommand{\B}{\mathcal{B}}

\begin{document}

\title{Dyck paths, binary words, and Grassmannian permutations avoiding an increasing pattern}

\author{Krishna Menon}
\address{Department of Mathematics, Chennai Mathematical Institute, India}
\email{krishnamenon@cmi.ac.in}

\author{Anurag Singh}
\address{Department of Mathematics, Indian Institute of Technology (IIT) Bhilai, India}
\email{anurags@iitbhilai.ac.in}

\maketitle

\begin{abstract}
    A permutation is called {\it Grassmannian} if it has at most one descent. 
    The study of pattern avoidance in such permutations was initiated by Gil and Tomasko in 2021. 
    We continue this work by studying Grassmannian permutations that avoid an increasing pattern. 
    In particular, we count the Grassmannian permutations of size $m$ avoiding the identity permutation of size $k$, thus solving a conjecture made by Weiner. 
    We also refine our counts to special classes such as odd Grassmannian permutations and Grassmannian involutions. 
    We prove most of our results by relating Grassmannian permutations to Dyck paths and binary words.
\end{abstract}

\section{Introduction}

For $n\geq 1$, let $\pi =\pi_1\cdots \pi_n$ be the one-line representation of a permutation of the set $[n]=\{1,\dots,n\}$. 
For $n \geq m \geq 1$, a permutation $\sigma= \sigma_1\cdots \sigma_n$ {\it contains} a permutation (or pattern) $\pi=\pi_1\cdots \pi_m$ if there exists a subsequence $1\leq h(1)< h(2)< \dots < h(m)\leq n$ such that for any $i,j\in[m]$, $\sigma_{h(i)}<\sigma_{h(j)}$ if and only if $\pi_i<\pi_j$. 
We say that the permutation $\sigma$ {\itshape avoids} $\pi$ if it does not contain $\pi$. 

The study of pattern avoidance in permutations was initiated by Knuth \cite{knu}, and the work of Simion and Schmidt \cite{SimSch}  was the first one to focus solely on enumerative results.  Since then many authors have studied pattern avoidance for various combinatorial objects. This includes the study of pattern avoidance in binary trees \cite{bintree2, bintree}, rooted forests \cite{genforests, forests},  circular permutations \cite{ callan, vella}, Dyck paths \cite{dyckpaths}, set partitions \cite{setavoidance2, setavoidance1} and compositions \cite{mansourbook}. Pattern avoidance has also been studied for its applications to algebraic geometry (see, for example \cite{TE2018, sara, JZJ2022, Christian20225, WooYong}).

Very recently, Gil and Tomasko \cite{grass, grassodd} initiated the study of pattern avoidance in Grassmannian permutations. 
In particular, they have shown that all non-identity Grassmannian permutations are Wilf-equivalent and obtained expressions for the number of Grassmannian permutations of a given size avoiding a non-identity Grassmannian permutation. 
They also count the number of Grassmannian permutations of size $2k - 2$ and $2k - 3$ that avoid the identity permutation of size $k$ (denoted $\operatorname{id}_k$). 
In this article, we build on their work by studying Grassmannian permutations of arbitrary size avoiding $\operatorname{id}_k$. 
We also refine our results to some special classes of Grassmannian permutations.

The outline of the article is as follows. 
In \Cref{prelim}, we describe a convenient representation of Grassmannian permutations using binary words that will help in proving results in the sequel. 
In \Cref{countsec}, we count the number of Grassmannian permutations of size $m$ avoiding $\operatorname{id}_k$. 
We first do this using binary words and recursions, thus proving a conjecture by Michael Weiner (see \Cref{count}). 
We also use Dyck paths to obtain a different expression for these numbers. 
We then refine this result in \Cref{parsec} by studying avoidance of an increasing pattern in odd and even Grassmannian permutations. 
Finally, we obtain similar counts for other classes of Grassmannian permutations in \Cref{othersec}.

\section{Preliminaries}\label{prelim}

We denote the set of Grassmannian permutations of $[n]$ by $\G_n$. 
We use binary words to encode these permutations. 
Let $w = w_1 w_2 \cdots w_n$ be a binary word of length $n$. 
We construct the Grassmannian permutation $G(w)$ as follows. 
If $A = \{i \in [n] \mid w_i = 0\}$ is a set of size $k$, then we set the first $k$ terms of $G(w)$ to be those of $A$ listed in increasing order. 
The remaining $n - k$ terms are those of $[n] \setminus A$ listed in increasing order.

\begin{example}
The Grassmannian permutation associated with the binary word $0^3101^30 = 000101110$ is $123594678 \in \G_9$.
\end{example}

The following result is an immediate consequence of the definitions.

\begin{proposition}
    Each permutation in $\G_n$ is of the form $G(w)$ for some binary word $w$ of length $n$. 
    This representation is unique for any non-identity permutation, and the binary words that correspond to the identity permutation are those of the form $0^j1^{n - j}$ for $j \in [0, n]$.
\end{proposition}

As a warm-up exercise, we count the Grassmannian permutations according to the number of fixed points. 
A fixed point in a permutation $\pi$ of $[n]$ is a number $i \in [n]$ such that $\pi(i) = i$.

\begin{proposition}
    The number of Grassmannian permutations of length $n$ with $k$ fixed points is $1$ if $k = n$, $0$ if $k = n - 1$ and $(k + 1)2^{n - k - 2}$ otherwise.
\end{proposition}

\begin{proof}
The case $k = n$ corresponds to the identity permutation, and it is clear that no permutation can have exactly $n - 1$ fixed points. 
Let $w$ be a binary word not of the form $0^i1^{n - i}$ for $i \in [0, n]$. 
This means that there is at least one $1$ that appears before a $0$. 
Hence, $w$ is of the form $0^a 1 w' 0 1^b$ for some $a, b \geq 0$ and binary word $w'$ of length $n - (a + b + 2)$. 
From the way the permutation $G(w)$ is defined, it can be checked that the number of fixed points is $a + b$, which are $\{1, 2, \ldots, a, n - b + 1, n - b + 2, \ldots, n\}$. 
So, if we want $G(w)$ to have $k$ fixed points, we have to choose $a \in [0, k]$ and a binary word $w'$ of length $n - k - 2$. 
This gives us the required result.
\end{proof}

We now turn to pattern avoidance. 
We say a binary word $w'$ \emph{contains} a binary word $w$ if it contains $w$ as a subsequence. 
We say a binary word $w'$ \emph{avoids} $w$ if it does not contain $w$.

\begin{example}
The binary word $01001101100$ contains the pattern $1100$. 
One such instance is indicated in the following:
$0{\color{blue}\underline{1}}00{\color{blue}\underline{1}}1{\color{blue}\underline{0}}11{\color{blue}\underline{0}}0$. 
It avoids the pattern $001001$ since any pair of $1$s that have at least two $0$s between them must use the first $1$.
\end{example}

It can be checked that pattern avoidance in Grassmannian permutations translates to binary words as described in the following proposition.

\begin{proposition}\label{bincont}
    If $G(w)$ is not the identity permutation, $G(w')$ contains $G(w)$ if and only if $w'$ contains $w$. 
    If $G(w)$ is the identity permutation of size $k$, then $G(w')$ contains $G(w)$ if and only if $w'$ contains $0^j1^{k - j}$ for some $j \in [0, k]$.
\end{proposition}

For any binary word $w$, denote by $\G_n(w)$ the set of Grassmannian permutations of length $n$ that avoid $G(w)$. 
We reprove a result from \cite{grass} using the language of binary words.

\begin{theorem}[{\cite[Theorem 3.3]{grass}}]
    If $w$ is a binary word of length $k$ and $G(w)$ is not the identity permutation,
    \begin{equation*}
        |\G_n(w)| = 1 + \sum_{j = 2}^{k - 1} \binom{n}{j}.
    \end{equation*}
\end{theorem}

\begin{proof}
The proof is in the same lines as that of \cite[Proposition 3.22]{4k}. 
By \Cref{bincont}, we have to count the binary words $v$ of length $n$ that do not contain the subsequence $w = w_1 w_2 \cdots w_k$. 
Suppose $j \in [0, k - 1]$ is the largest number such that $v$ contains $w_1 w_2 \cdots w_j$. 
Then, using the \emph{left most} occurrence of $w_1w_2 \cdots w_j$, we can see that $v$ must be of the form
$$(1-w_1)^{i_1}\ w_1\ (1-w_2)^{i_2}\ w_2\ \cdots\ (1-w_j)^{i_j}\ w_j\ (1 - w_{j+1})^{i_{j+1}},$$
where $(i_1, \dots , i_{j+1})$ is a sequence in $\mathbb{Z}_{\geq 0}$ and $i_1 + \cdots + i_{j+1} = n - j$. 
There are $\binom{n}{j}$ such words of length $n$. 
Since all words corresponding to the identity permutation avoid $w$, removing the over-counting, we get
\begin{equation*}
    |\G_n(w)| = \sum_{j = 0}^{k - 1} \binom{n}{j} - n = 1 + \sum_{j = 2}^{k - 1} \binom{n}{j}.
\end{equation*}
\end{proof}

As we make use of Dyck paths in the sequel, we now set up relevant notations. 
A \emph{Dyck path of semilength $n$} is a lattice path that starts at the origin, ends at $(2n, 0)$, has steps $U = (1, 1)$ and $D = (1, -1)$, and never falls below the $x$-axis. 
A \emph{peak} in a Dyck path is an up-step immediately followed by a down-step. 
The height of a peak is the $y$-coordinate of the point at the end of its up-step. 
Similarly, a \emph{valley} is a down-step immediately followed by an up-step and its height is the $y$-coordinate of the point at the end of its down-step. 
The number of Dyck paths of semilength $n$ is the $n^{th}$ Catalan number (for example, see \cite{sta_cat}) given by
\begin{equation*}
    C_n = \frac{1}{n + 1}\binom{2n}{n}.
\end{equation*}

For any $n \geq 1$ and $k \geq 0$, the number of Dyck paths of semilength $n + 1$ whose last peak is of height $n + 1 - k$ is given by
\begin{equation*}
    T(n, k) = \frac{n - k + 1}{n + 1} \cdot \binom{n + k}{n}.
\end{equation*}
A proof of this can be found, for example, in \cite{aval2008multivariate}. 
These numbers are sometimes called the \emph{ballot numbers} and are listed as \cite[\href{https://oeis.org/A009766}{A009766}]{oeis}.

\section{Counting Grassmannian permutations avoiding $\operatorname{id}_k$}\label{countsec}

In this section, we prove an expression conjectured by Michael Weiner, stated in \cite[Page 4]{grass}, for the number of Grassmannian permutations of size $m$ avoiding $\operatorname{id}_k$. 
We do this using binary words and recursions. 
We also obtain a different expression for the same numbers by representing the binary words as Dyck paths.

\begin{theorem}[{\cite[Conjecture]{grass}}]\label{count}
    For any $k \geq 2$ and $m \in [k, 2k - 2]$, we have
    \begin{equation}\label{conj}
        |\G_m(0^k)| = \sum_{j = 1}^{2k - m}(-1)^{j - 1} j \cdot \binom{2k - m - j}{j} \cdot C_{k - j}.
    \end{equation}
\end{theorem}

We note that the terms on the right-hand side of the above equation for those $j$ where $2k - m - j < j$ are $0$, and hence the sum only runs over $j \in [1, k - \lfloor m/2 \rfloor]$. 
But we write the sum as above to make the expressions in the following computations easier to read.

% For any $k \geq 1$ and $m \geq 0$, we set $\B(k, m)$ to be the set of binary words of length $m$ that avoid $0^j1^{k - j}$ for all $j \in [0, k]$. 
% We denote the cardinality of $\B(k, m)$ by $B(k, m)$. 
For $m \geq 0$, let $\mathcal{B}(m)$ denote the set of binary words of length $m$. 
Set $\mathcal{B}(0, m) = \emptyset$, and for $k \geq 1$, let
$$\mathcal{B}(k, m) = \{w \in \mathcal{B}(m) : w \text{ avoids } 0^j1^{k-j} \text{ for all } j \in [0, k]\}.$$
Moreover, we let $B(k, m)$ denote the cardinality of $\mathcal{B}(k, m)$.

Note that if $m \geq k$, for any $i \in [0, m]$, $0^i1^{m - i}$ contains $0^j1^{k - j}$ for some $j \in [0, k]$. 
Hence, the results in \Cref{prelim} imply that $B(k, m) = |\G_m(0^k)|$ for $m \geq k \geq 2$. 
We show that $B(k, m)$ is defined by the following recurrence:
\begin{enumerate}[label=(\roman*)]
    \item $B(0, m) = 0$ for all $m \geq 0$ and $B(k, 0) = 1$ for all $k \geq 1$.

    \item If $m \geq 2k - 1$, then $B(k, m) = 0$.

    \item For all other values of $k, m \geq 1$, we have
    \begin{equation*}
        B(k, m) = B(k, m - 1) + B(k - 1, m - 1) - T(k - 1, m - k).
    \end{equation*}
\end{enumerate}

The first two points are easy to see. 
We now prove the point (iii). 
Let $w = w_1 w_2 \cdots w_m$ be a binary word of length $m$. 
Set $w' =  w_1 w_2 \cdots w_{m - 1}$. 
We have the following:
\begin{itemize}
    \item $w \in \B(k, m)$ with $w_m = 1$ if and only if $w' \in \B(k - 1, m - 1)$.
    
    \item $w \in \B(k, m)$ with $w_m = 0$ if and only if $w' \in \B(k ,m - 1)$ and does not have $(k - 1)$ $0$s.
\end{itemize}

Binary words in $\B(k, m - 1)$ that have $(k - 1)$ $0$s are of the form
\begin{equation*}
    1^{a_{k - 1}} 0 1^{a_{k - 2}} 0 \cdots 1^{a_1} 0
\end{equation*}
for some sequence $(a_1, \ldots, a_{k - 1})$ in $\mathbb{Z}_{\geq 0}$ such that $a_1 + \cdots + a_i \leq i$ for all $i \in [k - 1]$ and $a_1 + \cdots + a_{k - 1} = m - k$. 
Associating the Dyck path given by
\begin{equation*}
    U D^{a_1} U D^{a_2} \cdots D^{a_{k - 1}} U D^{2k - m}
\end{equation*}
to such a sequence shows that they are counted by $T(k - 1, m - k)$. 
This gives us the required recursion.

For any $k, m \geq 0$, we define $A(k, m)$ to be the right-hand side of \eqref{conj}, i.e.,
\begin{equation*}
    A(k, m) = \sum_{j = 1}^{2k - m}(-1)^{j - 1} j \cdot \binom{2k - m - j}{j} \cdot C_{k - j}.
\end{equation*}
We show that it satisfies the same recurrence as $B(k, m)$. 
The fact that $A(0, m) = 0$ for all $m \geq 0$ and that $A(k, m) = 0$ if $m \geq 2k$ follows from the definition of $A(m, k)$. 
To prove that $A(k, 0) = 1$ for all $k \geq 1$, we have to show that for all $k \geq 1$,
\begin{equation*}
    \sum_{j = 1}^k (-1)^{j - 1} j \cdot \binom{2k - j}{j} \cdot C_{k - j} = 1.
\end{equation*}
Simplifying the summands and re-indexing, this is equivalent to proving that
\begin{equation*}
    \sum_{j = 0}^n (-1)^{j} \binom{n + 1}{n - j}\binom{n + j + 1}{j} = (-1)^n
\end{equation*}
for any $n \geq 0$. 
This can be proved using the binomial theorem for negative powers by considering the generating function equality
\begin{equation*}
    (1 + x)^{n + 1} \cdot (1 + x)^{-(n + 2)} = \frac{1}{1 + x}.
\end{equation*}

Proving the analogue of point (iii) for $A(k, m)$ boils down to proving that
\begin{equation*}
    T(k - 1, m - k) = \sum_{j = 0}^{2k - m - 1} (-1)^{j} \binom{2k - m - 1 - j}{j} \cdot C_{k - 1 - j}
\end{equation*}
for all $k \geq 1$ and $m \in [2k - 1]$. 
This is a consequence of the following lemma, which is given as Equation (5) in \cite{lang} (set $a = n + k - 1$ and $b = k$ to obtain the expression given in \cite{lang}).

\begin{lemma}[{\cite[Equation (5)]{lang}}]\label{balascatlem}
    For any $a, b \geq 0$ with $b \in [-a, a]$, we have
    \begin{equation}\label{balascat}
        T(a, b) = \sum_{j = 0}^{a - b} (-1)^j \binom{a - b - j}{j} \cdot C_{a - j}.
    \end{equation}
\end{lemma}

% \begin{proof}
% The lemma follows by using a defining recursion for these values of $T(a, b)$, which is
% \begin{enumerate}
%     \item $T(a, a) = C_a$ for all $a \geq 0$,
%     \item $T(a, -a) = 0$ for all $a \geq 1$, and
%     \item $T(a, b) = T(a - 1, b) + T(a, b - 1)$ for all other values of $a, b$.
% \end{enumerate}
% The right-hand side of \eqref{balascat} also satisfies this recursion. 
% Points (1) and (3) are straightforward to prove. 
% For point (2), we note that when $b = -a$ the terms of the sum on the right-hand side are a signed version of the table of numbers listed as \cite[\href{https://oeis.org/A060693}{A060693}]{oeis}. 
% Point (2) now follows since the alternating sums of the rows of \href{https://oeis.org/A060693}{A060693} are known to be $0$.
% \end{proof}

This shows that $A(k, m) = B(k, m)$ for all $k, m \geq 0$ and hence proves \Cref{count}.

\subsection{Count using Dyck paths}

In this sub-section, we count the binary words in $\B(k, m)$ using Dyck paths. 

\begin{lemma}\label{binasdyck}
    For any $k, m \geq 1$, $B(k, m)$ is the number of Dyck paths of semilength $(k + 1)$ where the sum of the heights of the first and last peak is $(2k - m)$.
\end{lemma}

\begin{proof}
The binary words in $\B(k, m)$ with $j$ $0$s are of the form
\begin{equation*}
    1^{a_j} 0 1^{a_{j - 1}} 0 \cdots 1^{a_2} 0 1^{a_1} 0 1^{a_0}
\end{equation*}
where $(a_0, a_1, \ldots, a_j)$ is a sequence in $\mathbb{Z}_{\geq 0}$ such that
\begin{itemize}
    \item $j \in [0, k - 1]$,
    \item $a_0 + a_1 + \cdots a_i < (k - j) + i$ for all $i \in [0, j]$, and
    \item $j + a_0 + a_1 + \cdots a_j = m$.
\end{itemize}
We associate the Dyck path given by
\begin{equation*}
    U^{k - j} D^{a_0 + 1} U D^{a_{1}} \cdots U D^{a_{j - 1}} U D^{a_j} U D^{k + j - m}
\end{equation*}
to such a sequence. 
This gives us the required result.
\end{proof}

As a consequence of the above lemma and \Cref{count}, we obtain the following result which gives an expression for the numbers listed as \cite[\href{https://oeis.org/A114503}{A114503}]{oeis}.

\begin{proposition}\label{dycksumpeak}
    The number of Dyck paths of semilength $n$ where $s \leq 2n - 2$ is the sum of the heights of the first and last peaks is
    \begin{equation*}
        \sum_{j = 1}^{\lfloor s/2 \rfloor} (-1)^{j - 1} j \binom{s - j}{j} C_{n - 1 - j}.
    \end{equation*}
\end{proposition}

We now use \Cref{binasdyck} to obtain an alternate expression for $B(k, m)$. 
To do this we use a refinement of the ballot numbers.

\begin{lemma}
    The number of Dyck paths of semilength $n + 1$ with first peak of height $a$ and last peak of height $b$ where $a + b \leq 2n$ is given by
    \begin{equation*}
        \binom{2n - a - b}{n - a} - \binom{2n - a - b}{n}.
    \end{equation*}
\end{lemma}

\begin{proof}
This lemma is an easy consequence of the result proved in \cite{4587979}. 
The Dyck paths we want to count correspond to paths from $(a + 1, a - 1)$ to $(2n - b + 1, b - 1)$ using the steps $U = (1, 1)$ and $D = (1, -1)$ that do not fall below the $x$-axis. 
These can be counted using the reflection principle (for example, see \cite{Hilton1991CatalanNT}).
\end{proof}

\begin{theorem}
    For any $k, m \geq 1$, we have
    \begin{equation*}
        |\G_m(0^k)| = B(k, m) = \sum_{a = 1}^{2k - m - 1} \left[\binom{m}{k - a} - \binom{m}{k}\right].
    \end{equation*}
\end{theorem}

Studying the bijection in \Cref{binasdyck} and the representation of Grassmannian permutations using binary words, we get the following results.

\begin{corollary}\label{totbinfix0}
    The number of binary words (of any length, including the empty word) that avoid $0^i1^{k - i}$ for all $i \in [0, k]$ and have exactly $j$ $0$s is the ballot number $T(k, j + 1)$.
\end{corollary}

\begin{corollary}\label{totbinav}
For any $k \geq 1$, the number of binary words that avoid $0^i1^{k - i}$ for all $i \in [0, k]$ is
\begin{equation*}
    \sum_{m = 0}^{2k - 2} B(k, m) = C_{k + 1} - 1.
\end{equation*}
\end{corollary}

\begin{corollary}
The number of Grassmannian permutations that avoid $\operatorname{id}_k$ is
\begin{equation*}
    \sum_{m = 0}^{2k - 2}|\G_m(0^k)| = C_{k + 1} - \binom{k}{2} - 1.
\end{equation*}
\end{corollary}

\section{Parity restrictions}\label{parsec}

A permutation is said to be \emph{odd} if it has an odd number of inversions (occurrences of the pattern $21$). 
We define a binary word $w$ to be odd if the corresponding permutation $G(w)$ is odd. 
We have the following characterization of odd binary words.

\begin{proposition}
    If $w$ is the binary word
    \begin{equation*}
        1^{a_{k}} 0 1^{a_{k - 1}} 0 \cdots 1^{a_1} 0 1^{a_0}
    \end{equation*}
    then the number of inversions in the permutation $G(w)$ is $\sum\limits_{i = 1}^{k} i \cdot a_i$. 
    In particular, $w$ is odd if and only if an odd number of terms in the sequence $(a_1, a_3, a_5, \ldots)$ are odd.
\end{proposition}

\begin{proof}
We need to count the number of occurrences of the pattern $21 = G(10)$ in the Grassmannian permutation $G(w)$. 
This is just the number of times $10$ appears as a subsequence of $w$. 
Hence, the number of inversions contributed by a $0$ in $w$ is the number of $1$s before it. 
This gives us the required expression for the number of inversions.
\end{proof}

\begin{remark}
As a consequence of the above result, we obtain a generating function for Grassmannian permutations keeping track of the number of inversions. 
Using $x$ to keep track of the length of the permutation and $t$ for the number of inversions, the required generating function is
\begin{equation*}
    \left(\frac{1}{1 - x}\right) \left[1 + \sum_{k \geq 1} \left( \frac{x}{1 - xt}\right) \left( \frac{x}{1 - xt^2}\right) \cdots \left( \frac{x}{1 - xt^k}\right) \right] - \frac{x}{(1 - x)^2}.
\end{equation*}
The last term removes the over-counting for the identity permutations.
\end{remark}

We study odd and even Grassmannian permutations that avoid $\operatorname{id}_k$. 
We use $O(k, m)$ to denote the number of odd binary words in $\B(k, m)$ and similarly define $E(k, m)$. 
Note that in particular, we have
\begin{equation*}
    B(k, m) = O(k, m) + E(k, m)
\end{equation*}
and that $O(k, m)$ is the number of odd permutations in $\G_m(0^k)$. 
Since we have already obtained expressions for $B(k, m)$ in \Cref{countsec}, we only list expressions for $O(k, m)$.

In the following results, we set $C_n$ to be $0$ if $n$ is not an integer. 
Similar to \cite[Proposition 3.1]{grass}, we have the following.

\begin{proposition}\label{paritycatalan}
    For any $k \geq 2$, we have
    \begin{equation*}
        O(k, 2k - 2) = \dfrac{C_{k - 1} + C_{(k - 2)/2}}{2}.
    \end{equation*}
    Also, $O(k, 2k - 3) = 2E(k, 2k - 2)$.
\end{proposition}

Before proving this result, we need a small lemma.

\begin{lemma}\label{alleven}
    The number of Dyck paths of semilength $n$ that have all peaks and valleys at odd height is $C_{(n - 1)/2}$.
\end{lemma}

\begin{proof}
Since the first peak should be at odd height, the first string of up-steps should be of odd length. 
The string of down-steps following it should be of even length since the first valley should be at odd height. 
Continuing this logic, we see that such Dyck paths are those of the form
\begin{equation*}
    U^{2a_1 + 1} D^{2b_1} \cdots U^{2a_k} D^{2b_k + 1}.
\end{equation*}
This shows that the semilength $n$ should be odd. 
Also, note that such a Dyck path is primitive (i.e., it only touches the $x$-axis at $(0, 0)$ and $(2n, 0)$). 
Hence, associating the Dyck path of semilength $(n - 1)/2$ given by
\begin{equation*}
    U^{a_1} D^{b_1} \cdots U^{a_k} D^{b_k}
\end{equation*}
to this Dyck path gives the required count.
\end{proof}

\begin{proof}[Proof of \Cref{paritycatalan}]
Just as in the proof of \cite[Proposition 3.1]{grass}, we use Dyck paths to represent the words in $\B(k, 2k - 2)$. 
Notice that any word in $\B(k, 2k - 2)$ is of the form
\begin{equation*}
    1^{a_{k -1}} 0 1^{a_{k - 2}} 0 \cdots 1^{a_2} 0 1^{a_1} 0
\end{equation*}
where $(a_1, a_2, \ldots, a_k)$ is a sequence in $\mathbb{Z}_{\geq  0}$ such that $a_1 + a_2 + \cdots + a_i \leq i$ for all $i \in [k - 1]$ and $a_1 + a_2 + \cdots + a_{k - 1} = k - 1$. 
These sequences correspond to Dyck paths of semilength $(k - 1)$ by setting $a_i$ to be the number of down-steps immediately following the $i^{th}$ up-step:
\begin{equation*}
    U D^{a_1} U D^{a_2} \cdots U D^{a_{k - 1}}.
\end{equation*}
An \emph{odd Dyck path} is one where, when written in the above form, an odd number of terms in $(a_1, a_3, a_5, \ldots)$ are odd. 
We define even Dyck paths to be those that are not odd.

Note that $O(k, 2k - 2)$ is the number of odd Dyck paths of semilength $(k - 1)$ and $E(k, 2k - 2)$ is the number of even Dyck paths of semilength $(k - 1)$. 
The total number of Dyck paths of semilength $(k - 1)$ is $C_{k - 1}$. 
Hence, we can prove the expression for $O(k, 2k - 2)$ by showing that there are $C_{(k - 2)/2}$ more odd Dyck paths than even ones.

We first consider Dyck paths that have at least one peak or valley at even height. 
For such a Dyck path, find the first peak or valley at even height. 
If it is a peak, change it to a valley and if it is a valley, change it to a peak.

\begin{example}
The first peak or valley at even height in the Dyck path on the left in \Cref{changeeven} is the first valley, which is at height $0$. 
Changing this to a peak gives the Dyck path on the right. 
Note that the change in the sequence $(a_1, a_2, a_3, a_4)$ corresponding to the Dyck path is as follows:
\begin{equation*}
    (1, 0, 1, 2) \rightarrow (1 - 1, 0 + 1, 1, 2) = (0, 1, 1, 2).
\end{equation*}
Hence, the Dyck path on the left is even and the one on the right is odd.
\end{example}

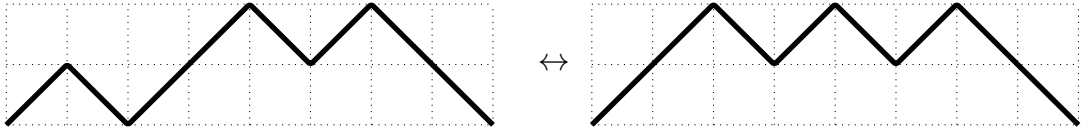
\begin{figure}[H]
    \centering
    \begin{tikzpicture}[scale=0.8]
    \draw[dotted] (0, 0) grid (8, 2);
    \draw[rounded corners=1, color=black, line width=2] (0, 0) -- (1, 1) -- (2, 0) -- (3, 1) -- (4, 2) -- (5, 1) -- (6, 2) -- (7, 1) -- (8, 0);
    % \node [circle, draw = black, fill = blue, inner sep = 2pt] at (2, 0) {};
    \node at (9, 1) {$\leftrightarrow$};
    \end{tikzpicture}
    \begin{tikzpicture}[scale=0.8]
    \draw[dotted] (0, 0) grid (8, 2);
    \draw[rounded corners=1, color=black, line width=2] (0, 0) -- (1, 1) -- (2, 2) -- (3, 1) -- (4, 2) -- (5, 1) -- (6, 2) -- (7, 1) -- (8, 0);
    % \node [circle, draw = black, fill = blue, inner sep = 2pt] at (2, 2) {};
    \end{tikzpicture}
    \caption{Changing the first peak or valley at even height.}
    \label{changeeven}
\end{figure}

It can be checked that this gives a bijection between odd and even Dyck paths that have at least one peak or valley at even height.

We now show that Dyck paths where all peaks and valleys are at odd heights must be odd. 
This will then prove our required result by \Cref{alleven}. 
Suppose the Dyck path
\begin{equation*}
    U D^{a_1} U D^{a_2} \cdots U D^{a_{k - 1}}
\end{equation*}
has all peaks and valleys at odd heights. 
If $a_i \neq 0$ and $i < (k - 1)$, then there is a peak after the $i^{th}$ up-step of height $i - (a_1 + a_2 + \cdots + a_{i - 1})$ and a valley before the $(i + 1)^{th}$ up-step of height $i - (a_1 + a_2 + \cdots + a_{i})$. 
This shows that all $a_i$ for $i < (k - 1)$ are even. 
If $k$ is odd, then $a_1 + a_2 + \cdots + a_{k - 1} = k - 1$ and hence $a_{k - 1}$ is even. 
This means that the last peak is of even height $a_{k - 1}$, which is a contradiction. 
This means that $k$ is even, $a_{k - 1}$ is odd, and hence the Dyck path is odd. 
This proves the first part of the proposition.

To prove the second statement in the proposition, we use the same method to study $\B(k, 2k - 3)$ as in \cite[Proposition 3.1]{grass}. 
The binary words in $\B(k, 2k - 3)$ that
\begin{itemize}
    \item have $(k - 2)$ $0$s are in bijection with the words in $\B(k, 2k - 2)$ via adding a $0$ at the end of the word, and
    \item those that have $(k - 2)$ $1$s are in bijection with the words in $\B(k, 2k - 2)$ via adding a $1$ at the start of the word.
\end{itemize}
Both these actions change the parity of the word if and only if $k$ is even. 
This shows that $O(k, 2k - 3) = 2O(k, 2k - 2)$ if $k$ is odd and $O(k, 2k - 3) = 2E(k, 2k - 2)$ if $k$ is even. 
But the expression for $O(k, 2k - 2)$ shows that $O(k, 2k - 2) = E(k, 2k - 2)$ when $k$ is odd. 
This proves the second part of the proposition.
\end{proof}

\begin{remark}
The expression for $O(k, 2k - 2)$ can also be derived using recursions just as in \cite[Proposition 1]{SimSch}. 
The numbers $O(k, 2k - 2)$ are listed as \cite[\href{https://oeis.org/A007595}{A007595}]{oeis} and $E(k, 2k - 2)$ are listed as \cite[\href{https://oeis.org/A000150}{A000150}]{oeis}. 
The number $C_{k - 1} = B(k, 2k - 2)$ coincides with the number of permutations of length $(k - 1)$ that avoid the pattern $132$. 
Similarly, the number $O(k, 2k - 2)$ coincides with the number of even permutations of length $(k - 1)$ that avoid the pattern $132$ and $E(k, 2k - 2)$ coincides with the number of such odd permutations. 
The numbers $O(k, 2k - 2)$ also count the Dyck paths of semilength $(k - 1)$ that have an even number of peaks at even height. 
A bijection between odd Dyck paths and such Dyck paths can be obtained using the same ideas as in the proof of \Cref{paritycatalan}.
\end{remark}

The idea in the proof of \Cref{paritycatalan} can be generalized to obtain an expression for $O(k, m)$ in terms of $B(k, m)$.

\begin{theorem}\label{oddcount}
    For any $k, m \geq 1$, we have
    \begin{equation*}
        2O(k, m) = 
        \begin{cases}
            B(k, m) + B(\frac{k}{2}, \frac{m - 2}{2}) - B(\frac{k}{2}, \frac{m}{2}) - B(\frac{k - 2}{2}, \frac{m - 2}{2}), &\text{if both $k$ and $m$ are even}\\
            B(k, m) - 2B(\left\lfloor \frac{k}{2} \right\rfloor, \left\lfloor \frac{m - 1}{2} \right\rfloor), &\text{otherwise.}
        \end{cases}
    \end{equation*}
\end{theorem}

%\commKM{I think using Dyck paths as in \Cref{binasdyck} instead of lattice paths in this proof will be confusing since the notion of odd Dyck paths are already used in the previous proof and in the general proof depending on the parity of $k$ and $j$, it might be that even Dyck paths correspond to odd binary words. Added a remark also.}

\begin{proof}
For $j \in [0, k - 1]$, the words in $\B(k, m)$ with $j$ $0$s are of the form
\begin{equation*}
    1^{a_j} 0 1^{a_{j - 1}} 0 \cdots 1^{a_2} 0 1^{a_1} 0 1^{a_0}
\end{equation*}
where $j + a_0 + a_1 + \cdots + a_j = m$ and $(j - i) + a_0 + a_1 + \cdots + a_i \leq k - 1$ for all $i \in [0, j]$. 
We associate the following lattice path to such a binary word:
% The steps for these lattice paths are $U$ and $D$ and start at the origin, just as in the Dyck path case. 
% The lattice path we associate the the binary word \eqref{binword} is
\begin{equation}\label{latticepath}
    D^{a_0} U D^{a_1} U D^{a_2} \cdots U D^{a_j}.
\end{equation}
These are lattice paths that start at the origin, have $j$ up-steps, $m - j$ down-steps, and do not fall below the line $y = j - k + 1$. 
We use $\B(k, m)$ to denote the set of these lattice paths as well. 
Such a lattice path is called \emph{odd} if an odd number of terms in $(a_1, a_3, a_5, \ldots)$ are odd. 
The remaining lattice paths are called even. 
Note that $O(k, m)$ is the number of odd lattice paths in $\B(k, m)$.

\begin{remark}
We could have used Dyck paths, just as in \Cref{binasdyck}, to represent such binary words. 
However, under this bijection, an odd Dyck path (in the sense mentioned in \Cref{paritycatalan}) might correspond to an even binary word depending on the parity of $k$ and $j$. 
To avoid this confusion, we use these lattice paths.
\end{remark}

We deal with the case when $k$ is odd and $m$ is even, the others are similar. 
We have to show that there are $2B(\frac{k - 1}{2}, \frac{m - 2}{2})$ more even lattice paths than odd ones. 
We first consider lattice paths that have a peak or valley whose height has the same parity as $j - k + 1$ (which, since $k$ is odd, is the parity of $j$). 
For such lattice paths, changing the first such peak or valley gives a bijection between odd and even lattice paths.

\begin{example}
\Cref{latpath} shows two lattice paths that are matched by this bijection corresponding to the binary words $110011$ and $110101$. 
Here $k = 5, m = 6,$ and $j = 2$.
\end{example}

\begin{figure}[H]
    \centering
    \begin{tikzpicture}[scale=1, xscale = -1]
    \draw[dotted] (0, 0) grid (6, 2);
    \draw[rounded corners=1, color=black, line width=2] (0, 0) -- (1, 1) -- (2, 2) -- (3, 1) -- (4, 0) -- (5, 1) -- (6, 2);
    \node at (-0.75, 1) {$\leftrightarrow$};
    \end{tikzpicture}
    \begin{tikzpicture}[scale=1, xscale = -1]
    \draw[dotted] (0, 0) grid (6, 2);
    \draw[rounded corners=1, color=black, line width=2] (0, 0) -- (1, 1) -- (2, 2) -- (3, 1) -- (4, 2) -- (5, 1) -- (6, 2);
    \end{tikzpicture}
    \caption{Matching lattice paths of opposite parity.}
    \label{latpath}
\end{figure}
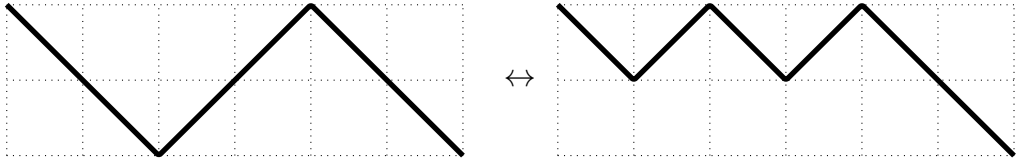

It is easy to see that if $j$ is odd, or if $j$ is even and $a_0 > 0$ is even, then a lattice path of type \eqref{latticepath} must have at least one peak or valley whose height has the same parity as $j$. 
So, if a lattice path in $\B(k, m)$ does not have a peak or valley whose height has the same parity as $j$, then we must have that
\begin{itemize}
    \item $j$ is even,
    \item $a_0 = 0$ or $a_0$ is odd, and
    \item $a_1, \dots , a_{j-1}$ are even (which means that the path must be even).
\end{itemize}
In addition, since $m$ is assumed to be even, the first and third items imply that $a_0 + a_j$ must be even, so $a_0$ and $a_j$ have the same parity.
% The remaining lattice paths in $\B(k, m)$ are those of the form \eqref{latticepath} where the following hold:
% \begin{itemize}
%     \item If $i \in [j - 1]$, then $a_i$ is even.
%     \item If $a_0 \neq 0$, then it has the same parity as $j + 1$.
%     \item If $a_j \neq 0$, then it has the same parity as $j + 1$.
% \end{itemize}
% Since $m$, the total number of steps, is even, the above conditions imply that $j$ must be even. 
% Similarly, exactly one of $a_0$ or $a_j$ being non-zero is also not possible. 
% Note that this means that all the remaining lattice paths are even.

If both $a_0$ and $a_j$ are non-zero, then the lattice path is of the form
\begin{equation*}
    D^{2b_0 + 1} U^{2c_1} D^{2b_1} U^{2c_2} \cdots U^{2c_p} D^{2b_p + 1}.
\end{equation*}
Associating the lattice path given by
\begin{equation*}
    D^{b_0} U^{c_1} D^{b_1} \cdots U^{c_p} D^{b_p}
\end{equation*}
to such a lattice path gives a bijection with the lattice paths in $\B(\frac{k - 1}{2}, \frac{m - 2}{2})$.

If both $a_0, a_j = 0$, then the lattice path is of the form
\begin{equation*}
    U^{2c_1 + 1} D^{2b_1} U^{2c_2} \cdots D^{2b_{p - 1}} U^{2c_p + 1}.
\end{equation*}
Associating the lattice path given by
\begin{equation*}
    U^{c_1} D^{b_1} U^{c_2} \cdots D^{b_{p - 1}} U^{c_p}
\end{equation*}
to such a lattice path gives a bijection with lattice paths in $\B(\frac{k - 1}{2}, \frac{m - 2}{2})$.

The above observations show that there are $2B(\frac{k - 1}{2}, \frac{m - 2}{2})$ more even lattice paths than odd ones in $\B(k, m)$ and hence proves the expression for $O(k, m)$ when $k$ is odd and $m$ is even.

The same method works for the other cases as well. 
When $k$ and $m$ are both even, there are some odd lattice paths that do not have a peak or valley at height having the same parity as $j - k + 1$. 
In fact, precisely $B(\frac{k}{2}, \frac{m - 2}{2})$ of them are odd and $B(\frac{k}{2}, \frac{m}{2}) + B(\frac{k - 2}{2}, \frac{m - 2}{2})$ of them are even. 
This is why the expression in this case is slightly different.
\end{proof}

From the proof of the above result and \Cref{totbinfix0,totbinav}, we obtain the following.

\begin{corollary}
The number of odd binary words that avoid $0^i1^{k - i}$ for all $i \in [0, k]$ and have exactly $j$ $0$s is
\begin{align*}
    \tfrac{1}{2}(T(k, j + 1) - 2T(\tfrac{k}{2}, \tfrac{j + 2}{2})), &\text{ if both $k$ and $j$ are even,}\\
    \tfrac{1}{2}(T(k, j + 1) - 2T(\tfrac{k - 1}{2}, \tfrac{j + 2}{2}) - T(\tfrac{k - 1}{2}, \tfrac{j}{2})), &\text{ if $k$ is odd and $j$ is even, and}\\
    \tfrac{1}{2}(T(k, j + 1) - T(\lfloor\tfrac{k - 1}{2}\rfloor, \tfrac{j + 1}{2})), &\text{ otherwise.}
\end{align*}
\end{corollary}

\begin{corollary}
The number of odd Grassmannian permutations avoiding $\operatorname{id}_k$ is
\begin{equation*}
    \frac{C_{k + 1}}{2} - 2C_{\frac{k + 1}{2}} + 1
\end{equation*}
when $k$ is odd and if $k$ is even it is
\begin{equation*}
    \frac{C_{k + 1} - C_{k/2}}{2} - C_{\frac{k + 2}{2}} + 1.
\end{equation*}
\end{corollary}

\section{Other restrictions}\label{othersec}

In this section, we consider the avoidance of $\operatorname{id}_k$ in special classes of Grassmannian permutations mentioned in \cite[Section 2]{grass}.

\subsection{BiGrassmannian permutations}

A permutation $\pi$ is said to be \emph{biGrassmannian} if both $\pi$ and $\pi^{-1}$ are Grassmannian. 
Rewriting a result from \cite{grass} in terms of binary words, we have the following.

\begin{proposition}[{\cite[Proposition 2.1]{grass}}]\label{bigrassprop}
    The biGrassmannian permutations of size $m$ are those in $\G_m(1010)$ and are counted by
    \begin{equation*}
        1 + \binom{m + 1}{3}.
    \end{equation*}
\end{proposition}

We study biGrassmannian permutations avoiding $\operatorname{id}_k$.

\begin{theorem}\label{bigrass}
    For $k\leq m < 2k$, the number of biGrassmannian permutations of size $m$ that avoid $\operatorname{id}_k$ is
    % \begin{equation*}
    %     1 + \binom{m + 1}{3}
    % \end{equation*}
    % if $m < k$ and if  it is
    \begin{equation*}
        \binom{2k - m + 1}{3}.
    \end{equation*}
\end{theorem}

\begin{proof}
% If $m < k$, then all the biGrassmannian permutations of size $m$ avoid the identity permutation of size $k$. 
% Hence, the first part of the result follows from \Cref{bigrassprop}.

When $m = k$, the only biGrassmannian permutation of size $k$ that contains $\operatorname{id}_k$ is the identity itself. 
Using \Cref{bigrassprop}, we get that there are
\begin{equation*}
    \binom{k + 1}{3}
\end{equation*}
biGrassmannian permutations of size $k$ that avoid $\operatorname{id}_k$.

When $m \geq k$, the binary words corresponding to biGrassmannian permutations of size $m$ that avoid $\operatorname{id}_k$ are of the following forms:
\begin{enumerate}[label=(\roman*)]
    \item $0^a1^b0^c$ for some $a, b, c \geq 1$ such that $a + b \leq k - 1$ and $a + c \leq k - 1$.
    
    \item $0^a1^b0^c1^d$ for some $a, b, c, d \geq 1$ such that $a + b + d \leq k - 1$ and $a + c + d \leq k - 1$.
    
    \item $1^a0^b$ for some $a, b \geq 1$ such that $a \leq k - 1$ and $b \leq k - 1$.
    
    \item $1^a0^b1^c$ for some $a, b, c \geq 1$ such that $a + c \leq k - 1$ and $b + c \leq k - 1$.
\end{enumerate}

If $m > k$, we show that the binary words listed above are in bijection with the non-identity biGrassmannian permutations of size $2k - m$. 
This will then prove the second part of the result. 
Since $m > k$, words of type (i) cannot have $b \leq m - k$ or $c \leq m - k$. 
We replace $b$ by $b - (m - k)$ and $c$ by $c - (m - k)$. 
We make the same replacement for words of type (ii). 
For words of type (iii) and (iv) we replace $a$ by $a - (m - k)$ and $b$ by $b - (m - k)$. 
It can be checked that this gives us the required bijection.
\end{proof}

We now consider odd biGrassmannian permutations avoiding $\operatorname{id}_k$. 
Since biGrassmannian permutations are precisely those Grassmannian permutations that avoid the pattern $2413$ \cite[Proposition 2.1]{grass}, we have the following from \cite{grassodd}.

\begin{proposition}[{\cite[Theorem 3.3]{grassodd}}]
    Set $a(m)$ to be the number of odd biGrassmannian permutations of size $m$. 
    Then,
    \begin{equation*}
        a(m) =
        \begin{cases}
            \frac{1}{4} \binom{m + 2}{3}, & \text{if $m$ is even}\\
            \frac{1}{24}(m - 1)(m + 1)(m + 3), & \text{if $m$ is odd.}
        \end{cases}
    \end{equation*}
    % \begin{equation*}
    %     \frac{1}{4} \binom{m + 2}{3}
    % \end{equation*}
    % if $m$ is even and if $m$ is odd it is
    % \begin{equation*}
    %     \frac{1}{24}(m - 1)(m + 1)(m + 3).
    % \end{equation*}
\end{proposition}

\begin{theorem}\label{bigrassparity}
    % Set $a(m)$ to be the number of odd biGrassmannian permutations of size $m$. 
    The number of odd biGrassmannian permutations of size $m$ that avoid $\operatorname{id}_k$ is
    \begin{align*}
        a(m), &\text{ if $m \leq k$,}\\
        a(2k - m), &\text{ if $m > k$ and $(m - k)$ is even, and}\\
        a(2k - m - 2), &\text{ if $m > k$ and $(m - k)$ is odd.}
    \end{align*}
\end{theorem}

\begin{proof}
If $m \leq k$, then all odd biGrassmannian permutations avoid $\operatorname{id}_k$. 
Also, if $m \geq 2k$, then all odd biGrassmannian permutations contain $\operatorname{id}_k$. 
The proof for the other two cases are similar to the proof of \Cref{bigrass}. 
For example, suppose that $m > k$, $(m - k)$ is odd, and
\begin{equation*}
    0^a 1^b 0^c 1^d
\end{equation*}
is an odd biGrassmannian permutation avoiding $\operatorname{id}_k$. 
The Grassmannian permutation associated to the binary word
\begin{equation*}
    0^a 1^{b - (m - k - 1)} 0^{c - (m - k - 1)} 1^d
\end{equation*}
is an odd biGrassmannian permutation of length $(2k - m - 2)$. 
Similar operations on other odd biGrassmannian permutations of size $m$ avoiding $\operatorname{id}_k$ show that they correspond to odd biGrassmannian permutations of size $(2k - m - 2)$. 
A similar technique works for proving the result when $m > k$ and $(m - k)$ is even.
\end{proof}

\subsection{Grassmannian involutions}

These are Grassmannian permutations $\pi$ that satisfy $\pi^{-1} = \pi$. 
Rewriting a result from \cite{grass} in terms of binary words, we have the following characterization of Grassmannian involutions.

\begin{proposition}[{\cite[Proposition 2.3]{grass}}]
    The Grassmannian involutions are those of the form $G(0^{k_1}1^{k_2}0^{k_2}1^{k_3})$ for some $k_1, k_2, k_3 \geq 0$ and the number of those of size $m$ is
    \begin{equation*}
        \left\lceil \frac{m^2 + 1}{4} \right\rceil.
    \end{equation*}
\end{proposition}

The following result can be proved using a similar bijection to the one described in the proof of \Cref{bigrass}.

\begin{theorem}
    For $k \leq m < 2k$, the number of Grassmannian involutions of size $m$ that avoid $\operatorname{id}_k$ is
    % \begin{equation*}
    %     \left\lceil \frac{m^2 + 1}{4} \right\rceil
    % \end{equation*}
    % if $m < k$ and when  it is
    \begin{equation*}
        \left\lfloor \frac{(2k - m)^2}{4} \right\rfloor.
    \end{equation*}
\end{theorem}

Just as we did for biGrassmannian permutations, we now study odd Grassmannian involutions avoiding $\operatorname{id}_k$.

\begin{proposition}
    Set $b(m)$ to be the number of odd Grassmannian involutions of size $m$. 
    Then,
    \begin{equation*}
        b(m) = \left\lfloor \frac{(m + 1)^2}{8} \right\rfloor.
    \end{equation*}
\end{proposition}

\begin{proof}
For $m \leq 4$, direct calculations show that $b(m) = m - 1$. 
We now show that for $m \geq 5$, we have
\begin{equation*}
    b(m) = b(m - 4) + m - 1.
\end{equation*}
This will prove the required result.

The binary words corresponding to odd Grassmannian involutions of size $m$ are of the following forms:
\begin{enumerate}[label=(\roman*)]
    \item $0^a 1^b 0^b$ where $a, b \geq 1$ and $b$ is odd.
    
    \item $0^a 1^b 0^b 1^c$ where $a, b, c \geq 1$ and $b$ is odd.
    
    \item $1^a 0 ^a$ where $a \geq 1$ is odd.
    
    \item $1^a 0^a 1^b$ where $a, b \geq 1$ and $a$ is odd.
\end{enumerate}
In words of the first two types, if $b \neq 1$, replace $b$ by $(b - 2)$. 
Similarly, in words of the last two types, if $a \neq 1$, replace $a$ by $(a - 2)$. 
This gives a bijection between such words and odd Grassmannian involutions of size $(m - 4)$. 
The number of remaining binary words is $(m - 1)$. 
This proves the required recursion.
\end{proof}

Using the same ideas as in the proof of \Cref{bigrassparity}, we have the following result.

\begin{theorem}
    The number of odd Grassmannian involutions of size $m$ that avoid the $\operatorname{id}_k$ is
    \begin{align*}
        b(m), &\text{ if $m \leq k$,}\\
        b(2k - m), &\text{ if $m > k$ and $(m - k)$ is even, and}\\
        b(2k - m - 2), &\text{ if $m > k$ and $(m - k)$ is odd.}
    \end{align*}
\end{theorem}

\section{Concluding remarks}

It would be interesting to see if \Cref{count} (or equivalently, \Cref{dycksumpeak}) and \Cref{balascatlem} can be proved directly, possibly by the Principle of Inclusion-Exclusion, instead of using recursions. 
We list some particular cases of these identities which might be easier to tackle than the general results.

\begin{enumerate}[label=(\roman*)]
    \item For $0 \leq m < k$, we have
    \begin{equation*}
        \sum_{j = 1}^{k}(-1)^{j - 1} j \cdot \binom{2k - m - j}{j} \cdot C_{k - j} = 2^m.
    \end{equation*}
    
    \item For $k \geq 1$, we have
    \begin{equation*}
        \sum_{j = 1}^{k}(-1)^{j - 1} j \cdot \binom{k - j}{j} \cdot C_{k - j} = 2^k - k - 1.
    \end{equation*}
\end{enumerate}

\section{Acknowledgements}
We would like to thank Michael Weiner for his comments on an earlier version of this paper. 
We would also like to thank the referees for their careful reading and helpful suggestions. 
The computer algebra system SageMath \cite{Sage} provided valuable assistance in studying examples. 
This work was done during the first author's visit to the Indian Institute of Technology Bhilai, funded by SERB, India, via the project  AV/VRI/2022/0140. Krishna Menon is partially supported by a grant from the Infosys Foundation, and Anurag Singh is partially supported by the Research Initiation Grant from the Indian Institute of Technology Bhilai (No. 2009301).

\bibliographystyle{abbrv}

\end{document}